\documentclass[a4paper,10pt,reqno, english]{amsart}  
\usepackage[utf8]{inputenc}
\usepackage[T1]{fontenc}
\usepackage{amsmath,amsthm}
\usepackage{amsfonts,amssymb,enumerate}
\usepackage{url,paralist}
\usepackage{mathtools}  
\usepackage[colorlinks=true,urlcolor=blue,linkcolor=red,citecolor=magenta]{hyperref}
\usepackage{enumerate}
\usepackage{anysize}
\usepackage{tikz}

\theoremstyle{plain}
\newtheorem{thm}{Theorem}[section]
\newtheorem{lem}[thm]{Lemma}
\newtheorem{cor}[thm]{Corollary}

\newtheorem{conj}[thm]{Conjecture}

\newtheorem*{thm*}{Theorem}

\theoremstyle{definition}
\newtheorem{defn}[thm]{Definition}
\newtheorem{ex}[thm]{Example}

\newcommand{\R}{\mathbb{R}}
\newcommand{\dist}{\mathrm{dist}}
\newcommand{\lk}{\mathrm{lk}}

\begin{document}

\title[Fair division and Sperner's lemma]{Fair division and generalizations of Sperner-\\ and KKM-type results}



\author[M. Asada \and F. Frick \and V. Pisharody \and M. Polevy \and D. Stoner \and L.H. Tsang \and Z. Wellner]{Megumi Asada \and Florian Frick \and Vivek Pisharody \and Maxwell Polevy \and\\ David Stoner \and Ling Hei Tsang \and Zoe Wellner}

\address[MA]{Department of Mathematics and Statistics, Williams College, Williamstown, MA 01267, USA}
\email{maa2@williams.edu}

\address[FF, VP, ZW]{Department of Mathematics, Cornell University, Ithaca, NY 14853, USA}
\email{\{ff238, vap52, zaw5\}@cornell.edu}

\address[MP]{Department of Mathematics, Northeastern University, Boston, MA 02115, USA}
\email{polevy.m@husky.neu.edu}

\address[DS]{Mathematics Department, Harvard University, 1 Oxford St, Cambridge, MA 02138, USA}
\email{dstoner@college.harvard.edu}

\address[LHT]{Department of Mathematics, The Chinese University of Hong Kong}
\email{henryt918@gmail.com}

\date{\today}
\maketitle


\begin{abstract}
\small
We treat problems of fair division, their various interconnections, 
and their relations to Sperner's lemma and the KKM theorem as well as their variants. We
prove extensions of Alon's necklace splitting result in certain regimes and relate it to hyperplane mass partitions. 
We show the existence of fair cake division and rental harmony in the sense of Su even in the absence of full
information. Furthermore, we extend Sperner's lemma and the KKM theorem to (colorful) quantitative versions
for polytopes and pseudomanifolds. For simplicial polytopes our results turn out to be improvements
over the earlier work of De Loera, Peterson, and Su on a polytopal version of Sperner's lemma. 
Moreover, our results extend the work of Musin on quantitative Sperner-type results for PL manifolds.
\end{abstract}


\section{Introduction}

The area of fair division comprises a class of mathematical problems where topology has found
some of its most striking applications. Usually one strives to partition an object and distribute it
among a set of agents under a certain fairness requirement. Examples include equipartitions of
measures in Euclidean space by affine hyperplanes, cutting up a necklace with $q$ types of beads
and partitioning the pieces into $k$ parts each containing the same number of beads of each kind,
and dividing a unit interval cake into $n$ pieces such that $n$ cake eaters will not be jealous of 
one another given their subjective preferences about the cake.
Here we treat these three incarnations of the fair division problem and their interconnections
as well as results from topological combinatorics that can be used to establish the existence
of fair divisions: the KKM theorem and Sperner's lemma. 

Section~\ref{sec:necklaces} is concerned with the necklace splitting problem. Two thieves can cut an unclasped
necklace with $q$ different types of beads in $q$ places in order to fairly divide the beads of each kind between
them; see Goldberg and West~\cite{goldberg1985}. Alon~\cite{alon1987} extended this to the case of $k$ thieves.
Here $(k-1)q$ cuts suffice to accomplish a fair division among the $k$ thieves.
We establish optimal necklace splitting results with additional constraints. That is, we require that adjacent
pieces of the necklace cannot be claimed by certain pairs of thieves while the optimal number of $(k-1)q$ cuts
still suffices to accomplish this division of the necklace.
In particular, we show that $k = 4$ thieves can pass the necklace around in a circle (where a change of direction
is allowed) while still fairly dividing it with $3q$ cuts, that is, two pairs of thieves (the diagonal pairs) will not receive adjacent pieces;
see Theorem~\ref{thm:cyclic-4-splitting}:
\begin{thm*}
	For $k=4$ thieves and a necklace with $q$ types of beads, there exists a cyclic $k$-splitting of size at most~$(k-1)q = 3q$.  
\end{thm*}
This properly extends Alon's result for $k = 4$ and any~$q$. The proof uses hyperplane mass partitions in Euclidean space
by placing the necklace along the moment curve, similar to Matou\v sek's proof~\cite{matousek2008} of the $k=2$ case.
We also extend Alon's result
for arbitrary~$k$ and $q=2$ types of beads with a combinatorial proof; see Theorem~\ref{thm:binary-splitting}:
\begin{thm*}
	Given a necklace with $q = 2$ kinds of beads and $k$ thieves, there exists a binary necklace splitting of size~${2(k-1)}$. 
\end{thm*}
Here binary necklace splitting means that we may place the $k$ thieves on the vertices of a hypercube of dimension~$\lceil \log_2 k \rceil$,
and require thieves who receive adjacent pieces of the necklace to be joined by an edge in the hypercube. Our interest
in this particular constraint on the necklace partition stems from the fact that an extension of Theorem~\ref{thm:binary-splitting}
to any $q$ and $k=2^t$ would imply that a conjecture of Ramos on hyperplane mass partitions is true for configurations
of measures along the moment curve. Since all lower bounds for this problem are obtained from configurations along the 
moment curve, see Avis~\cite{avis1984}, and upper bounds can be obtained from studying solutions for these configurations~\cite{blagojevic2015},
it is of particular importance to understand this special case.

Section~\ref{sec:KKM} deals with envy-free divisions of desirables -- the usual metaphor being cake --
and undesirables, such as rent, among hungry birthday guests or frugal roommates, respectively.
While for necklace splittings what constitutes a fair division is an objective fact, the attraction of envy-free cake and rent divisions
is due to participants having subjective preferences: for a given partition of the cake, people might prefer different pieces. 
Su~\cite{su1999} showed (partially reporting on work of Simmons) that under mild conditions such envy-free divisions 
always exist: a cake can be cut into $n$ pieces such that $n$ people with subjective preferences get their most preferred piece,
and thus are not envious of anyone else.
Our goal is to establish existence results for envy-free divisions (of cake or rent) even in the absence
of full information, extending the work of Su. We show (see Corollary~\ref{cor:secret-pref} and Corollary~\ref{cor:dual-secret-pref}):
\begin{thm*}
	Envy-free cake divisions exist for any number of people even if the preferences of one person are secret.
	Dually, envy-free rent divisions exist for any number of people even if the preferences of one person are secret.
\end{thm*}

We note the practical importance of this existence theorem: if a birthday cake is to be divided into pieces before it is presented
to the birthday girl, but she gets the first pick, this can always be done in such a way that the guests and host of the party are
not envious of one another. Dually, of $n$ future roommates it suffices if only $n-1$ of them are involved in deciding how to divide the rent among
the rooms such that none of them will be jealous. To prove this theorem we strengthen the colorful KKM theorem
of Gale~\cite{gale1984}; see Theorem~\ref{thm:strong-colorful-KKM}.

In Section~\ref{sec:quantitative} we consider further extensions of the KKM theorem and Sperner's lemma, which was used by
Su~\cite{su1999} to establish his fair division theorem. For a simplicial subdivision of the $d$-simplex~$\Delta_d$, we consider
a labelling of its vertices with the $d+1$ vertices of~$\Delta_d$ in such a way that no vertex subdividing a face of $\Delta_d$ disjoint
from vertex $v$ of~$\Delta_d$ receives label~$v$. Given such a labelling, Sperner's lemma~\cite{sperner1928} guarantees the existence of a facet 
of the subdivision that exhibits all $d+1$ labels. De Loera, Peterson, and Su~\cite{deLoera2002} extended Sperner's lemma to a quantitative version
for arbitrary $d$-polytopes in place of the simplex~$\Delta_d$. Here quantitative means that for a subdivision of a $d$-polytope
on $n$ vertices, one can find at least $n-d$ facets exhibiting $d+1$ distinct labels. Slightly better bounds for a larger class of objects are due to Meunier~\cite{meunier2006}.
Musin~\cite{musin2014, musin2015} established such Sperner-type results for PL manifolds with polytopal boundary. 
We further extend Musin's results to pseudomanifolds with boundary. Moreover,
our lower bound for the number of facets with pairwise distinct labels depends on the number $f_{d-1}$ of $(d-1)$-faces in the boundary instead of
the number of vertices: the bound of $\frac{f_{d-1}-2}{d-1}$ fully labelled facets is at least as good as $n-d$ for simplicial polytopes 
(and more generally pseudomanifolds) by the lower bound theorem. It is a proper improvement apart from the case of stacked polytopes.

We define the notion of a $d$-pseudomanifold with boundary to be Sperner colored with respect to a closed $(d-1)$-pseudomanifold
to extend the corresponding definition for polytopes, and then show (see Theorem~\ref{thm:sperner-pseudomfld}):

\begin{thm*}
	Let $B$ be a $(d-1)$-pseudomanifold and $K$ a $d$-pseudomanifold with boundary such that $\partial K$ is a subdivision of~$B$. 
	Let $K$ be Sperner colored with respect to $B$. Then $K$ has at least $\frac{f_{d-1}(B)-2}{d-1}$ fully labelled facets.
\end{thm*}

While Musin's proof of Sperner-type results for certain classes of PL manifolds depends on the polytopal result of De Loera, Peterson, and
Su, our proof method does not. In particular, we give a new and independent proof of their result for simplicial polytopes, improve lower bounds,
and extend results to a larger class of objects. To prove our pseudomanifold version of Sperner's lemma, we extend methods of Ramos~\cite{ramos1996}
that he developed to prove hyperplane mass partition results. Let $f\colon M^{(d+1)} \longrightarrow N^{(d)}$ be a submersion of smooth manifolds,
then the level sets of $f$ are embedded $1$-submanifolds. Ramos uses the PL version of this preimage theorem to establish the nonexistence
of certain equivariant maps via path-following arguments. Since path-following is a standard way of proving Sperner-type results, we adapt
and extend Ramos' methods to be applicable in this situation. Moreover, we establish colorful extensions of our quantitative Sperner's lemma
for pseudomanifolds, see Theorem~\ref{thm:colored-sperner-pseudomfld}, and use it to prove a colorful KKM theorem for pseudomanifolds,
see Theorem~\ref{thm:colored-KKM-pseudomfld}.


\section*{Acknowledgements}

This research was performed during the \emph{Summer Program for Undergraduate Research} 2016 at Cornell University. 
The authors are grateful for the excellent research conditions provided by the program. 
The authors thank Henry Adams, Thomas B\aa\aa th, Jesus De Loera, Fr\'ed\'eric Meunier, Oleg Musin, and Edward Swartz for helpful discussions. 
MA is supported by a Clare Booth Luce scholarship. 
MP is supported by the Northeastern University Mathematics Department and the Northeastern University Scholars Program. 
DS is supported by a Watson--Brown Foundation scholarship. 
LHT is supported by the Mathematics Department, Science Faculty and Summer Undergraduate Research Programme 2016 of the Chinese University of Hong Kong. 


\section{Splitting necklaces with additional constraints}
\label{sec:necklaces}

An unclasped necklace with $q$ types of beads arranged in some arbitrary order is to be divided among
$k$ thieves in such a way that for each kind of bead each thief receives the same number of beads
of that kind, provided the number of beads of each kind is divisible by~$k$. This task is to be 
accomplished with as few cuts of the necklace as possible. Such a partition of the necklace among 
$k$ thieves is referred to as a \emph{$k$-splitting}. The number of cuts is the \emph{size} of a $k$-splitting.

For $k=2$ thieves Goldberg and West~\cite{goldberg1985} found that $q$ cuts always suffice and this number is optimal.
Alon and West~\cite{alon1986} conjectured that every open necklace with $q$ types of beads has a $k$-splitting of
size $(k-1)q$ and this was proven by Alon:

\begin{thm}[Alon~\cite{alon1987}]
\label{thm:necklaces}
	Every unclapsed necklace with $q$ types of beads and $ka_i$ beads of each type with
	$1 \le i \le q$ has a $k$-splitting of size at most~$(k-1)q$. 
\end{thm}

The $k=2$ case is a consequence of the Borsuk--Ulam theorem, and Theorem~\ref{thm:necklaces}
follows from more general topological machinery. Matou\v sek~\cite{matousek2008} observed that the $k=2$ case
can also be deduced from the ham sandwich theorem, one of the most elementary hyperplane equipartition
result: given $d$ probability measures $\mu_1, \dots, \mu_d$ on~$\R^d$ such that each affine hyperplane 
has measure zero, there exists an affine hyperplane $H$ with $\mu_i(H^+) = \mu_i(H^-)$ for all~$i$.
Here $H^+$ and $H^-$ denote the positive and negative halfspace of~$H$, respectively.
The essential idea is to place the necklace along the moment curve in~$\R^q$. The ham
sandwich theorem now guarantees the existence of a bisecting hyperplane. Any hyperplane in~$\R^q$
cuts the moment curve in at most $q$ points and cuts that happen to go through beads may be adjusted
in pairs to not pass through beads. Here we are interested in whether hyperplane equipartition results 
for more than one hyperplane may be applied to the necklace splitting problem in a similar way.

We give a brief overview of the sparse landscape of known results about hyperplane equipartitions and refer
to the survey~\cite{blagojevic2015}. A \emph{mass} is a finite Borel measure on~$\R^d$ such that each hyperplane has 
measure zero. A set of $k$ affine hyperplanes is an \emph{equipartition} of a mass if each of the 
$2^k$ orthants has the same measure. The least dimension~$d$ where $j$ masses on~$\R^d$ admit
a common equipartition by $k$ affine hyperplanes is denoted by~$\Delta(j,k)$. The ham sandwich 
theorem yields~${\Delta(j,1) \le j}$. Avis, for $k =1$, and then Ramos, for general~$k$, obtained the 
lower bound $\Delta(j, k) \ge \lceil j\frac{2^k-1}{k}\rceil$. Ramos conjectured this bound to be tight for all
$j$ and~$k$. No counterexamples to this conjecture are known. Upper bounds are due to Mani--Levitska,
Vre\'cica, and \v Zivaljevi\'c~\cite{mani2006}; Ramos~\cite{ramos1996}; Hadwiger~\cite{hadwiger1966}; 
and Blagojevi\'c, Haase, Ziegler, and the second author~\cite{blagojevic2015, blagojevic2016}. In fact, all known tight cases for Ramos' conjecture follow from the 
setup in~\cite{blagojevic2016}.

\begin{thm}[\cite{hadwiger1966, ramos1996, mani2006, blagojevic2016}]
	Ramos' conjecture $\Delta(j, k)=\lceil j\frac{2^k-1}{k}\rceil$ holds if $k = 1$, if $k = 2$ and
	$j \in \{2^t-1, 2^t, 2^t+1\}$, or if $k = 3$ and $j \in \{2,4\}$.
\end{thm}

Viewing necklace splittings under a mass equipartition framework places additional constraints on the way 
the necklace pieces are distributed. Placing the necklace on the moment curve as before, each thief occupies
an orthant of a hyperplane equipartition. However, orthants that are not contained in the same halfspace of
some hyperplane are not adjacent. For $k = 4$ thieves this leads to a \emph{cyclic} distribution of the pieces
of the necklace, that is, there are two disjoint pairs of thieves, say $1$ and $3$ as well as $2$ and~$4$, that 
may not receive adjacent necklace pieces. Here we consider two pieces to be nonadjacent even if only an 
empty piece is in between them. Empty pieces are the result of cutting twice at the same point.

\medskip
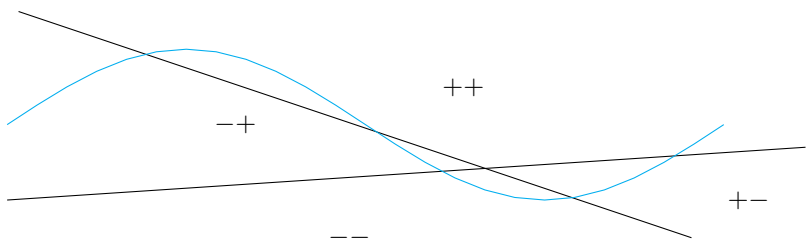
\begin{center}
\begin{tikzpicture}[xscale=1.5]
\draw (0,-1) -- (7,-0.3);
\draw (0.1, 1.5) -- (6, -1.5);
\draw [cyan,domain=0:2*pi] plot (\x, {sin(\x r)});
\node at (4, 0.5) {$++$};
\node at (6.5, -1) {$+-$};
\node at (3, -1.5) {$--$};
\node at (2, 0) {$-+$};
\end{tikzpicture}
\end{center}
\medskip

There is a particular complication that comes with using hyperplane equipartitions to prove necklace 
splitting results: we want to think of the beads as point masses, but the equipartition results only
hold for those masses that assign measure zero to each hyperplane. Thus, we rather think 
of beads as small $\varepsilon$-balls of mass (or intervals of length~$\varepsilon$ on the moment
curve). Then hyperplanes may pass through beads. Not 
wanting to cut beads into pieces, we first show that these cuts can be adjusted to not pass through 
beads. The case $k = 2$ was treated by Goldberg and West~\cite[Sec.~2]{goldberg1985}. 
It was pointed out to us by Fr\'ed\'eric Meunier that the case for general~$k$ was also treated by Alon, Moshkovitz, and 
Safra~\cite[Proof of Lemma~7]{alon2006} using the theory of network flows.

\begin{lem}
	Let $k \ge 2$ and $q \ge 1$ be integers and consider a necklace with $q$ types of beads and $ka_i$ beads of each type with
	$1 \le i \le q$ and $\sum_{i=1}^q ka_i=N$. 
 Let $f_1, \dots, f_q\colon [0,1] \longrightarrow \R$ be functions such that if the $m$-th 
	bead is of type $\ell \in \{1, \dots, q\}$, then $f_\ell(x) = 1$ for all $x \in [\frac{m-1}{N}, \frac{m}{N}]$ 
	and $f_\ell(x) = 0$ otherwise. Let the measure $\mu_i$ on $[0,1]$ be obtained from integrating~$f_i$.
	If there exist $(k-1)q$ cutpoints of $[0,1]$ and a partition of the resulting subintervals into $k$ parts
	such that each part has the same measure for each~$\mu_i$, then the necklace admits a $k$-splitting
	of size~$(k-1)q$ as well. Moreover, the distribution order among the $k$ parts is the same.
\end{lem}

\begin{proof}
We think of the measures $\mu_i$ as beads. The prerequisites then guarantee a $k$-splitting of the necklace
where cuts are allowed to pass through beads and each thief gets the same mass of each type of bead, but
the beads are not necessarily intact. We present an algorithm to adjust those cuts that pass through beads:
	\begin{compactenum}
		\item Find some cut which passes through a bead of color~$C$. 
		\item Consider the graph~$G$ constructed with the vertices corresponding to the thieves, that is, the
			parts of the splitting. For every cut through a color~$C$ bead, draw an edge between the two 
			thieves who receive the parts of nonzero length situated immediately on either side of that cut. 
			It is possible that some pairs of vertices might have multiple edges between them.
		\item Find a cycle in~$G$, and direct the edges along that cycle. For each directed edge~$A\to B$ 
			move the corresponding cut point so that~$B$'s part gets larger. Do this simultaneously for all 
			edges in the cycle, so that the total amount of bead~$C$ designated to each robber remains 
			constant. Continue this movement until either two cuts coincide, or one of the moving cuts 
			reaches a point in between two beads.
		\item If no cut passes through any bead, stop. Otherwise, return to step 1.
	\end{compactenum}
Each time step~$3$ runs, the quantity (number of pieces with nonzero length)$+$(number of cuts not located in between beads) 
decreases by one. It follows that this algorithm terminates. Once the algorithm terminates, the output is exactly a 
splitting with no extra cuts, and which has the same distribution order as the original splitting. Note that the existence of a 
cycle in the graph $G$ follows from the fact that no vertex in $G$ can have degree one; otherwise, some thief would have 
a noninteger amount of the color $C$ bead designated to them.
\end{proof}

\begin{thm}
\label{thm:cyclic-4-splitting}
	For $k=4$ thieves and a necklace with $q$ types of beads, there exists a cyclic $k$-splitting of size at most~$(k-1)q = 3q$.  
\end{thm}

\begin{proof}
We first prove the statement for $q$ a power of two, say~${q = 2^t}$. Place the necklace along the moment curve
$\gamma(x) = (x,x^2, \dots, x^d)$ in~$\R^d$ for $d = 3\cdot 2^{t-1}$. More precisely, let the $n$-th bead be at the 
point~$\gamma(n)$. Each kind of bead now defines a discrete mass on~$\R^d$ and thus there exists an equipartition
of these $q$ masses by two affine hyperplanes in dimension $d = 3\cdot 2^{t-1}$. Now associate the four orthants 
to thieves $1, \dots, 4$ in such a way that $1$ and $3$ as well as $2$ and $4$ receive opposite orthants, that is,
the orientations of both hyperplanes have to be flipped to move from one to the other orthant. This results in a cyclic
splitting of size~$3q$ since any hyperplane intersects the moment curve in at most $d = 3\cdot 2^{t-1}$ points.

Now let $q$ be arbitrary, say $q = 2^t-r$ for some integer~${r \ge 0}$. From the necklace with $q$ different kinds of beads
construct an extended necklace with $2^t$ kinds of beads by appending four beads each of $r$ new kinds. 
We know that a cyclic $4$-splitting of size $3\cdot 2^t$ exists for this extended necklace. Of those $3\cdot 2^t$ 
cuts $3r$ cuts have to split the appended $r$ kinds of beads. This leaves $3\cdot 2^t - 3r = 3q$ beads for the 
original part of the necklace. This induced splitting is cyclic as well.
\end{proof}

\begin{ex}
	In our notion of cyclic splitting a change of direction is allowed. Here we note that without this flexibility splittings with the optimal number of cuts
	that assign the pieces to thieves in order $1,2, \dots, k, 1, 2, \dots$ do not exist for any~$k \ge 3$.
\begin{figure}[h!]
\centering
\begin{tikzpicture}[xscale=1.5, yscale=1.5]
\draw (0,1) -- (3,1);
\draw [fill=green] (0.5,1) circle [radius=0.1];
\draw [fill=green] (0.75,1) circle [radius=0.1];
\draw [fill=green] (1,1) circle [radius=0.1];
\draw [fill=red] (1.25,1) circle [radius=0.1];
\draw [fill=red] (1.5,1) circle [radius=0.1];
\draw [fill=red] (1.75,1) circle [radius=0.1];
\draw [fill=green] (2,1) circle [radius=0.1];
\draw [fill=green] (2.25,1) circle [radius=0.1];
\draw [fill=green] (2.5,1) circle [radius=0.1];
\end{tikzpicture}
\caption{This necklace cannot be cut into five pieces and fairly distributed among three thieves in the order $1,2,3,1,2$.}
\end{figure}
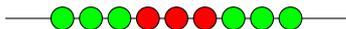

For $k =3$ thieves and $q = 2$ kinds of beads, the optimal number of cuts is $(k-1)q = 4$. In the necklace above two of those
cuts must divide the red block into three pieces and the other two cuts must occur in each of the green blocks. But with the
distribution order $1,2,3,1,2$, thief~$3$ does not receive any green beads. 

This reasoning generalizes to any odd number of thieves. For an even number of thieves, say $k = 4$, the following configuration
is a counterexample:
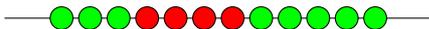
\begin{figure}[h!]
\centering
\begin{tikzpicture}[xscale=1.5, yscale=1.5]
\draw (0,1) -- (3.75,1);
\draw [fill=green] (0.5,1) circle [radius=0.1];
\draw [fill=green] (0.75,1) circle [radius=0.1];
\draw [fill=green] (1,1) circle [radius=0.1];
\draw [fill=red] (1.25,1) circle [radius=0.1];
\draw [fill=red] (1.5,1) circle [radius=0.1];
\draw [fill=red] (1.75,1) circle [radius=0.1];
\draw [fill=red] (2,1) circle [radius=0.1];
\draw [fill=green] (2.25,1) circle [radius=0.1];
\draw [fill=green] (2.5,1) circle [radius=0.1];
\draw [fill=green] (2.75,1) circle [radius=0.1];
\draw [fill=green] (3,1) circle [radius=0.1];
\draw [fill=green] (3.25,1) circle [radius=0.1];
\end{tikzpicture}
\caption{This necklace cannot be cut into seven pieces and fairly distributed among four thieves in the order $1,2,3,4,1,2,3$.}
\end{figure}

Here three cuts have to occur in the red block, one cut in the first green block, and the last two cuts in the second green block.
However, with the distribution order $1,2,3,4,1,2,3$, thief~$4$ receives no green beads.
\end{ex}

When the number of thieves $k$ is equal to $2^t$ for positive integer~$t$, the notion of a cyclic 
necklace splitting has a natural generalization. We may place each thief on the vertices of a 
$t$-dimensional hypercube, and stipulate that adjacent parts of the necklace partition are assigned 
to thieves along some edge of the hypercube. Such a splitting will be called a 
\emph{binary necklace splitting}; below is a more general definition for arbitrary~$k$. Again, cutting multiple times in the same spot and thus creating
empty pieces is explicitly allowed. 
If, for some $t\ge 3$, Ramos' conjecture $\Delta(j, t)=\lceil j\frac{2^t-1}{t}\rceil$
holds for arbitrarily large pairs $(j, t)$ with $t$ dividing~$j$, then the proof of Theorem~\ref{thm:cyclic-4-splitting} 
may be adapted to show that binary necklace splittings are always possible for $2^t$ thieves.

\begin{defn}
For some $k$ with $2^{t-1}<k\le 2^t$, we say that a given necklace partition 
is a \emph{binary necklace splitting} if there exists some assignment of the robbers to $k$ 
distinct length $t$ binary strings such that, for any pair of robbers that receive adjacent parts 
of the necklace, the corresponding binary strings differ by a single bit.
\end{defn}

For $k$ not a power of two, this generalizes the notion of a binary necklace splitting 
by stipulating that the thieves occupy some subset of the vertices of the smallest possible 
hypercube and if two thieves receive adjacent pieces they must share an edge in the hypercube.
The first relevant case 
here is $k=3$, with the condition implying that some robber gets every other piece of the necklace. 
See the figure below for an example:
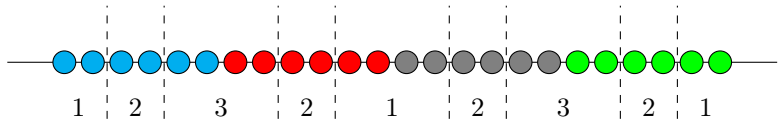
\begin{figure}[h!]
\centering
\begin{tikzpicture}[xscale=1.5, yscale=1.5]
\draw (0,1) -- (6.75,1);
\draw [fill=cyan] (0.5,1) circle [radius=0.1];
\draw [fill=cyan] (0.75,1) circle [radius=0.1];
\draw [fill=cyan] (1,1) circle [radius=0.1];
\draw [fill=cyan] (1.25,1) circle [radius=0.1];
\draw [fill=cyan] (1.5,1) circle [radius=0.1];
\draw [fill=cyan] (1.75,1) circle [radius=0.1];
\draw [fill=red] (3,1) circle [radius=0.1];
\draw [fill=red] (3.25,1) circle [radius=0.1];
\draw [fill=red] (2,1) circle [radius=0.1];
\draw [fill=red] (2.25,1) circle [radius=0.1];
\draw [fill=red] (2.5,1) circle [radius=0.1];
\draw [fill=red] (2.75,1) circle [radius=0.1];
\draw [fill=gray] (3.5,1) circle [radius=0.1];
\draw [fill=gray] (3.75,1) circle [radius=0.1];
\draw [fill=gray] (4,1) circle [radius=0.1];
\draw [fill=gray] (4.25,1) circle [radius=0.1];
\draw [fill=gray] (4.5,1) circle [radius=0.1];
\draw [fill=gray] (4.75,1) circle [radius=0.1];
\draw [fill=green] (5,1) circle [radius=0.1];
\draw [fill=green] (5.25,1) circle [radius=0.1];
\draw [fill=green] (5.5,1) circle [radius=0.1];
\draw [fill=green] (5.75,1) circle [radius=0.1];
\draw [fill=green] (6,1) circle [radius=0.1];
\draw [fill=green] (6.25,1) circle [radius=0.1];
\draw [dashed] (0.875, 0.5) -- (0.875, 1.5);
\draw [dashed] (1.375, 0.5) -- (1.375, 1.5);
\draw [dashed] (2.875, 0.5) -- (2.875, 1.5);
\draw [dashed] (2.375, 0.5) -- (2.375, 1.5);
\draw [dashed] (3.875, 0.5) -- (3.875, 1.5);
\draw [dashed] (4.375, 0.5) -- (4.375, 1.5);
\draw [dashed] (5.875, 0.5) -- (5.875, 1.5);
\draw [dashed] (5.375, 0.5) -- (5.375, 1.5);
\node (A) at (0.625, 0.6){1};
\node (B) at (1.125, 0.6){2};
\node (A) at (1.875, 0.6){3};
\node (A) at (2.625, 0.6){2};
\node (A) at (3.375, 0.6){1};
\node (A) at (4.125, 0.6){2};
\node (A) at (4.875, 0.6){3};
\node (A) at (5.625, 0.6){2};
\node (A) at (6.125, 0.6){1};
\end{tikzpicture}
\caption{A binary necklace partition for $k=3, q=4$}
\end{figure}

However, the $3$-robber case is a too restrictive to hope for the existence of a binary splitting for any configuration of more than
two kinds of beads, as we will point out in Example~\ref{ex:no-binary}.

We will now give a combinatorial proof that binary $k$-splittings exist for any number $k$ of thieves and two kinds of beads. 
The proof uses a sliding window trick that was already used by Goldberg and West~\cite{goldberg1985}, and Meunier~\cite{meunier2008}
in a similar context. Our proof below in particular is a combinatorial proof of Alon's result for two kinds of beads. Such combinatorial
proofs for $q=2$ were found before by Epping, Hochst\"attler, and Oertel~\cite{epping2004}; Meunier~\cite{meunier2008} treats
some additional cases. Epping, Hochst\"attler, and Oertel consider ``a paint shop problem'': in a factory cars of different car body types
and colors are built to order. The number of color changes between successive cars is to be minimized. Epping, Hochst\"attler, and Oertel,
unaware of Alon's earlier work on splitting necklaces, consider regular instances, where the number of cars of a certain body type and
color does not depend on body type or color. This is a necklace splitting problem where colors correspond to thieves and car body types 
correspond to kinds of beads. This problem is equivalent to necklace splitting if each car body type is colored equally often by each color
(but this number may depend on the body type). In this interpretation Theorem~\ref{thm:cyclic-4-splitting} shows that for four colors 
two pairs of disjoint color changes can be disallowed without increasing the upper bound for the number of color changes. 
The theorem below gives further constraints on the color changes we have to admit for any 
number of colors and two car body types.

\begin{thm}
\label{thm:binary-splitting}
	Given a necklace with $q = 2$ kinds of beads and $k$ thieves, there exists a binary necklace splitting of size~${2(k-1)}$. 
\end{thm}

\begin{proof}
To prove this result, we first rescale the necklace and place it along the interval~$[0, k]$, without loss of generality regarding the beads as 
equally sized intervals. We assume for simplicity that the necklace consists of black and white beads.
Given a necklace of length~$k$, a \emph{subnecklace} of length~$b$ for some integer $b \ge 0$ is defined as some 
length~$b$ subinterval $[t, t+b]$ of $[0, k]$, $0\le t\le k-b$. This subnecklace is called $\emph{balanced}$ if the amount 
of each color bead is precisely enough to satisfy $b$ thieves, that is, there is a fraction of $\frac{b}{k}$ of each color in this
subnecklace.

We now introduce two lemmas:

\begin{lem}
	If $b$ divides~$k$, then any length~$k$ necklace contains a length~$b$ balanced subnecklace.
\end{lem}

\begin{proof}
Consider the $\frac{k}{b}$ subnecklaces which lie along the intervals $[0, b],\ [b, 2b],\ \dots ,\ [k-b, k]$. 
If any of these necklaces is balanced, the statement is true already. Otherwise, some of these subnecklaces 
contain too many white beads and some contain too many black beads. It is impossible for all of the 
subnecklaces to be one of these types, since the subnecklaces on average have bead colors proportionally 
equal to that of the full necklace. Therefore, some two adjacent subnecklaces have too many of opposite 
types of beads. It follows by continuity that some subnecklace between these necklaces is balanced, as required.
\end{proof}

\begin{lem}
\label{lem:long-balanced}
	For any integer $b, \ 1\le b\le k-1$, any length~$k$ necklace contains either a length~$b$ or a length~${k-b}$ 
	balanced subnecklace.
\end{lem}

\begin{proof}
We may first line up $b$ identical copies of the necklace along $[0,k],\ [k, 2k],\ \dots ,\ [(b-1)k, bk]$. 
Regarding this as a single necklace, there is by the previous lemma some length~$b$ balanced 
subnecklace of this necklace. If this subnecklace is contained entirely in one of the length~$k$ 
necklaces, then we are done. Otherwise, the complement of the parts of the length~$k$ necklaces 
intersecting the balanced $b$ subnecklace is a balanced $k-b$ subnecklace.
\end{proof}

We now return to the proof of the theorem. We will prove a stronger hypothesis by strong induction 
on the number of thieves~$k$. In particular, we will prove that there always exists a binary splitting 
with $2(k-1)$ cuts for which the first and last pieces go to the same person. The base case $k=1$ 
is clear. 

Now suppose that for some $k\ge 2$, there is a length~$k$ necklace of two colors which must be 
split in a binary fashion among~$k$ thieves, with the first and last pieces going to the same person, 
and $2^{t-1}<k\le 2^t$. Since $k\ge 2$, we know that a binary necklace splitting exists for any necklace 
of length~$\lfloor \frac{k}{2}\rfloor$ or~$\lceil\frac{k}{2}\rceil$. By Lemma~\ref{lem:long-balanced}, 
there exists some length~$\lfloor \frac{k}{2}\rfloor$ or length~$\lceil\frac{k}{2}\rceil$ balanced subnecklace. 
Suppose the former exists; the case that the latter exists is analogous. Split the $k$ thieves into two groups; 
thieves $t_1, \ldots, t_{\lfloor \frac{k}{2}\rfloor}$ and thieves $u_1, \ldots, u_{\lceil\frac{k}{2}\rceil}$, and let 
$[b, b+\lfloor\frac{k}{2}\rfloor]$ be a balanced subnecklace with $0\le b\le \lceil\frac{k}{2}\rceil$. 
Then for $1\le i\le \lfloor \frac{k}{2}\rfloor$, we may assign a length $t-1$ binary string $b_i$ to thief 
$t_i$ such that the length~$b$ balanced subnecklace can be split using $2(\lfloor\frac{k}{2}\rfloor -1)$ 
cuts and assigned to these thieves in such a manner that adjacent pieces are assigned to thieves 
whose binary strings differ by one bit. We then combine the $[0, b]$ and $[b+\lfloor\frac{k}{2}\rfloor, k]$ 
parts of the necklace into a single length $\lceil\frac{k}{2}\rceil$ balanced subnecklace. We can similarly 
assign, for $1\le j\le \lceil\frac{k}{2}\rceil$, a length $t-1$ binary string $c_j$ to thief $u_j$ such that this length 
$k-b$ balanced subnecklace can be split using $2(\lceil\frac{k}{2}\rceil -1)$ cuts and assigned to these thieves 
in such a manner that adjacent pieces are assigned to thieves whose binary strings differ by one bit.

Now we claim that, using these $2+2(\lfloor\frac{k}{2}\rfloor -1)+2(\lceil\frac{k}{2}\rceil-1)$ cuts 
(including the cuts at $b$ and $b+\lfloor\frac{k}{2}\rfloor$), we may find a distribution and binary string assignment 
which satisfies the required conditions. Indeed, we may distribute the necklace parts exactly as they are distributed 
in the subnecklaces; this ensures that the same person gets the first and last pieces in the overall splitting. 
Furthermore, it ensures that there is exactly one adjacency between the $t$-thieves and the $u$-thieves; say between 
$t_x$ and $u_y$, which occurs twice.
This can be ensured even in the case that a cut in the $u$-subnecklace corresponds exactly to the location of the $t$-subnecklace, because in this case the extra cut available can be used to make an extra length $0$ piece on an appropriate end of the $t$-subnecklace. and assign this piece to a $u$-thief.
 
Let $||$ denote concatenation and $\oplus$ denote the bitwise XOR of binary strings,
that is, $0 \oplus 0 = 0 = 1 \oplus 1$ and $0 \oplus 1 = 1 = 1 \oplus 0$. Then we may assign $0||b_i$ to each thief~$t_i$, 
and $1||(c_j\oplus b_x\oplus c_y)$ to each thief~$u_j$. These $k$ binary strings all have length $t$ and are unique; it remains to check that every pair of adjacent pieces in the partition is assigned to thieves whose binary representations differ by one bit. Due to construction, this holds for all pairs of adjacent pieces assigned to thieves who are either both $t$-thieves or both $u$-thieves. Furthermore, the binary strings assigned to $t_x$ and $u_y$ are $0||b_x$ and $1||(c_y\oplus b_x\oplus c_y)=1||b_x$, which differ by one bit as required. This completes the induction, so the theorem holds as required.  
\end{proof}

\begin{ex}
\label{ex:no-binary}
	For three thieves a binary necklace splitting forces that every other piece goes to the same thief. For more than 
	two kinds of beads this condition is too restrictive to guarantee the existence of a binary necklace splitting for any
	order of the beads. Consider the following configuration:
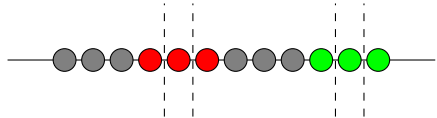
\begin{figure}[h!]
\centering
\begin{tikzpicture}[xscale=1.5, yscale=1.5]
\draw (0,1) -- (3.75,1);
\draw [fill=gray] (0.5,1) circle [radius=0.1];
\draw [fill=gray] (0.75,1) circle [radius=0.1];
\draw [fill=gray] (1,1) circle [radius=0.1];
\draw [fill=red] (1.25,1) circle [radius=0.1];
\draw [fill=red] (1.5,1) circle [radius=0.1];
\draw [fill=red] (1.75,1) circle [radius=0.1];
\draw [fill=gray] (2,1) circle [radius=0.1];
\draw [fill=gray] (2.25,1) circle [radius=0.1];
\draw [fill=gray] (2.5,1) circle [radius=0.1];
\draw [fill=green] (2.75,1) circle [radius=0.1];
\draw [fill=green] (3,1) circle [radius=0.1];
\draw [fill=green] (3.25,1) circle [radius=0.1];
\draw [dashed] (1.375, 0.5) -- (1.375, 1.5);
\draw [dashed] (1.625, 0.5) -- (1.625, 1.5);
\draw [dashed] (2.875, 0.5) -- (2.875, 1.5);
\draw [dashed] (3.125, 0.5) -- (3.125, 1.5);
\end{tikzpicture}
\caption{A configuration of beads without binary splitting for $k = 3$ thieves}
\end{figure}

The necklace above does not admit a binary necklace splitting among $k = 3$ thieves with the optimal number of $(k-1)q = 6$ cuts. 
Four of those cuts are forced (the dashed lines), as the red and green blocks have to be cut into three pieces each. Of the remaining two cuts one must
occur in each of the gray blocks since no thief receives three gray beads. Say thief $1$ and $3$ are not supposed to receive adjacent pieces.
Then both the central red and green bead have to go to thief~$2$, since otherwise thief~$2$ would receive two green or two red beads,
but there are two pieces of the necklace in between these beads, so they cannot both go to thief~$2$. This reasoning easily generalizes to $q \ge 3$
kinds of beads by appending blocks of three beads of the same kind.
\end{ex}

In summary, binary necklace splittings exist for $k = 4$ and any~$q$
and for any~$k$ and~$q \le 2$. They do not necessarily exist for $k = 3$ and~${q \ge 3}$.
Moreover, the validity of Ramos' conjecture would imply that 
binary splittings exist for $k$ a power of two and any~$q$. We thus conjecture:

\begin{conj}
	Given a necklace with $q$ kinds of beads and $k \ge 4$ thieves, there exists a binary necklace splitting of size~${(k-1)q}$. 
\end{conj}

For $k$ a power of two this conjecture is equivalent to the assertion that there are no counterexamples to
Ramos' conjecture along the moment curve, provided that $\log_2 k$ divides~$q$.


\section{KKM-type results and fair cake division}
\label{sec:KKM}

We now consider the following fair division problem: We wish to divide some rectangular ``cake'' into $d+1$ slices 
(indexed by~$i$) such that each person (indexed~$j$) can have their most preferred slice. That is, no person 
should be envious of another's piece --- an ``envy-free'' division. 

Being reasonable bakers, we bake our cake in a unit-length rectangular pan, place our finished cake along~$[0, 1]$, 
and slice perpendicular to this axis. Slicing the cake into a particular set of $d+1$ slices can be specified by giving the 
$(d+1)$-tuple $(x_0, \ldots, x_{d})$ of slice lengths, so that the slices are $(0, x_0), (x_0, x_0+x_1)$, etc. 
Note that $x_i \geq 0$ and $\sum_i x_i = 1$, so the $x_i$ are those points lying in the first orthant in the hyperplane 
specified by $\sum_i x_i = 1$, so the space of such divisions is the $d$-simplex~$\Delta_d$.

We will assume, in the language of Su~\cite{su1999}, that everyone is hungry, that is, any nonempty piece is preferable to any empty 
one, and that the preference sets --- subsets $C_i^{j} \subseteq \Delta_d$ of the space of possible divisions of the cake 
corresponding to the $j$-th individual preferring the $i$-th slice ---  are closed. Under these conditions an envy-free division
of the cake always exists: Su reports on a proof of Simmons that for any given preference sets $C^{i}_j \subseteq \Delta_d$,
$0 \le i,j \le d$, as above $d$ cuts of the unit interval cake suffice to distribute the cake among the $d+1$ individuals in an
envy-free manner. The core of the proof is an application of Sperner's lemma which we recall here. Let $K$ be a triangulation 
of~$\Delta_d$ on vertex set~$V$. Identify the vertices of $\Delta_d$ with $0, \dots, d$. A \emph{Sperner coloring} of~$K$ is a
map $c\colon V \longrightarrow \{0, \dots, d\}$ such that if vertex $v$ of~$K$ is contained in the face $\sigma$ of~$\Delta_d$
then $c(v) \in \sigma$. Here we think of $\{0, \dots, d\}$ as a set of $d+1$ colors.

\begin{lem}[Sperner's lemma~\cite{sperner1928}]
\label{lem:sperner}
	Let $K$ be a Sperner colored triangulation of~$\Delta_d$. Then there is a facet of~$K$ that exhibits all $d+1$ colors.
	In fact, there is always an odd number of such facets.
\end{lem}

Sperner's lemma and Brouwer's fixed point theorem, that any continuous self-map of the closed $d$-ball has a fixed point,
can be easily deduced from one another. We recall a third closely related result, the KKM theorem.
A \emph{KKM cover} of the $d$-simplex $\Delta_d$ is a collection of closed subsets 
$C_0, \dots, C_d \subseteq \Delta_d$ such that for any subset $I \subseteq \{0, \ldots, d\}$, 
the face of $\Delta_d$ determined by $I$ is covered by $\bigcup_{i \in I} C_i$.

\begin{thm}[Knaster, Kuratowski, Mazurkiewicz~\cite{knaster1929}]
\label{thm:kkm}
	Let $C_0, \dots, C_d \subseteq \Delta_d$ be a KKM cover. Then $\bigcap_i C_i \ne \emptyset$.
\end{thm}

This is simply due to the fact that $\bigcap_i C_i = \emptyset$ would imply that the map
$$f\colon \Delta_d \longrightarrow \Delta_d, x \mapsto \frac{1}{\sum_i \dist(x, C_i)}(\dist(x, C_0), \dots, \dist(x, C_d))$$
does not have a fixed point --- the KKM condition means that if $x$ is in the relative interior of the face $\sigma \subseteq \Delta_d$
then $f(x)$ is not in the relative interior of~$\sigma$ --- in contradiction to Brouwer's fixed point theorem.

Using partitions of unity subordinate to the covering and Birkhoff's theorem that any doubly stochastic matrix is a convex combination
of permutation matrices, Gale was able to extend the KKM theorem to the following colorful version:

\begin{thm}[Colorful KKM theorem, Gale~\cite{gale1984}]
\label{thm:colorful-KKM}
	Given $d+1$ KKM coverings $\{C^0_i\}_{i = 0, \dots, d}, \dots, \{C^d_i\}_{i = 0, \dots, d}$ of the $d$-simplex, 
	there exists some permutation $\pi$ of $\{0, \dots, d\}$ such that $\bigcap_{i = 0, \dots, d} C^i_{\pi(i)}\ne \emptyset$. 
\end{thm}

This theorem is ``colorful'' in the same sense that B\'ar\'any's extension of Carath\'eodory's theorem~\cite{barany1982} is a colorful Carath\'eodory's theorem.
We think of the $d+1$ KKM covers as having $d+1$ distinct colors. Gale's theorem guarantees an intersection among sets of pairwise
distinct colors. Choosing all $d+1$ KKM covers to be the same specializes to the classical KKM Theorem~\ref{thm:kkm}.

Theorem~\ref{thm:colorful-KKM} is precisely equivalent to the existence of envy-free cake divisions as defined above.
This was apparently first observed by Musin~\cite[Sec.~3]{musin2015-2}.
Here the assumption that individuals are hungry guarantees that the preference sets form a KKM cover.
Su~\cite{su1999} gave a proof of the existence of envy-free cake divisions, but was unaware of Gale's earlier work. 

\begin{cor}
\label{cor:div}
	Envy-free cake divisions exist for any number of people.
\end{cor}

\begin{proof}
	The configuration space of partitions of~$[0,1]$ into $d+1$ intervals is the $d$-simplex~$\Delta_d$. For each person~$j \in \{0,\dots,d\}$ 
	we have a KKM cover $C^j_0, \dots, C^j_d$ defined by $$C^j_i = \{x \in \Delta_d \: | \: \text{person} \ j \ \text{weakly prefers piece} \ i\}.$$ These
	sets are closed by assumption and they form a KKM cover since person $j$ will not prefer an empty piece over a nonempty one.
	Theorem~\ref{thm:colorful-KKM} now guarantees the existence of an envy-free cake division.
\end{proof}

The proof of Corollary~\ref{cor:div} presented in~\cite{su1999} uses Sperner's lemma and thus can easily be adapted to yield a combinatorial proof 
of Theorem~\ref{thm:colorful-KKM}, which we briefly sketch below. Bapat~\cite{bapat1989} also gave a combinatorial proof 
of Theorem~\ref{thm:colorful-KKM} using path-following methods. He established the following colorful Sperner's lemma.

\begin{thm}[Bapat~\cite{bapat1989}]
\label{thm:colored-sperner}
	Let $K$ be a triangulation of~$\Delta_d$ on vertex set~$V$. Let $c_0, \dots, c_d \colon V\longrightarrow \{0, \dots, d\}$
	be $d+1$ Sperner colorings of~$K$. Then there is a facet $\sigma$ of~$K$ and an ordering $v_0, \dots, v_d$ of its
	vertices such that the set $\{c_0(v_0), \dots, c_d(v_d)\}$ consists of all $d+1$ colors.
\end{thm}

In fact, Bapat shows that there are at least $(d+1)!$ pairs of facets and orderings of its vertices that exhibit all colors across the
$d+1$ Sperner colorings. We prove a quantitative extension of (a weaker version of) Theorem~\ref{thm:colored-sperner} to pseudomanifolds; see 
Theorem~\ref{thm:colored-sperner-pseudomfld}.

\begin{proof}[Proof of Theorem~\ref{thm:colorful-KKM}]
	We will show that for every $\varepsilon > 0$ there exists a permutation $\pi$ and a point $x \in \Delta_d$
	such that $x$ has distance at most $\varepsilon$ from every~$C^j_{\pi(j)}$. This implies the colorful KKM
	by letting $\varepsilon$ tend to zero since the simplex is compact and there are only finitely many permutations
	of the set~$\{0, \dots, d\}$.
	
	Let $K$ be a triangulation of~$\Delta_d$, where each face has diameter at most~$\varepsilon$. Denote by
	$K'$ the barycentric subdivision of~$K$. Color vertex $v$ of~$K'$ by color~$i$ if $v$ subdivides a face of
	dimension~$j$ in~$K'$ and~${v \in C_i^j}$. If for some vertex~$v$ this leaves multiple colors to choose from,
	decide for a color~$i$ arbitrarily provided that this choice does not violate the Sperner coloring condition.
	This is always possible. The existence of a facet exhibiting all colors guaranteed by Sperner's lemma now implies 
	our claim.
\end{proof}

Here we extend Gale's arguments to prove a stronger version of the colorful KKM theorem and apply it to fair cake division.

\begin{thm}[Strong colorful KKM theorem]
\label{thm:strong-colorful-KKM}
	Given $d$ KKM coverings $\{C^1_i\}_{i = 0, \dots, d}, \dots, \{C^d_i\}_{i = 0, \dots, d}$ of the $d$-simplex, 
	there is an $x \in \Delta_d$ and there are $d+1$ bijections $\pi_i\colon \{1,\dots,d\}\longrightarrow \{0,\dots,d\}\setminus\{i\}$, 
	$i \in \{0, \dots, d\}$, such that for all~$i$ we have $x \in C_{\pi_i(1)}^1\cap \dots\cap C_{\pi_i(d)}^d$.
\end{thm}

\begin{proof}
Let $\{C^1_i\}_{i = 0, \dots, d}, \dots, \{C^d_i\}_{i = 0, \dots, d}$ be open KKM covers. The case for closed KKM covers can be proved 
by considering the open sets $\{x\in \Delta_d \ | \ \dist(x,C_i^j)<\varepsilon\}$ and letting $\varepsilon$ tend to zero.
Without loss of generality, assume that $C_i^j$ is disjoint from the facet opposite the vertex~$i$.

Let
$$f_j\colon \Delta_d \longrightarrow \Delta_d, x \mapsto \frac{1}{\sum_k \dist(x, \Delta_d\setminus C_k^j)}(\dist(x, \Delta_d\setminus C_0^j), \dots, \dist(x, \Delta_d\setminus C_d^j))=(f_{j,0},\dots,f_{j,d})$$
and let $f\colon\Delta_d \longrightarrow \Delta_d$ be defined as $f=\frac1d \sum_{i=1}^{d} f_i$. 
The $i$-th coordinate $f_{j,i}$ of $f_j$ satisfies $f_{j,i}(x) > 0$ for a given $x \in \Delta_d$ if and only if $x \in C_i^j$, since $C_i^j$ is open.
Notice that $\sum_k \dist(x, \Delta_d\setminus C_k^j) > 0$ for every~$x \in \Delta_d$, since $\bigcup_k C_k^j = \Delta_d$.
Since $f(\sigma)\subseteq \sigma$ for any face $\sigma$ of~$\Delta_d$ the map~$f$ has degree one on the boundary~$\partial\Delta_d$. 
Thus there exists some $x_0$ such that $f(x_0)=\frac{1}{d+1} (1,\dots,1)$.

Now consider the $(d+1)\times (d+1)$ matrix 
$$M=
\begin{bmatrix}
    \frac{1}{d+1} & f_{1,0}(x_0)     & \dots & f_{d,0}(x_0)\\
    \frac{1}{d+1} & f_{1,1}(x_0)     & \dots & f_{d,1}(x_0)\\
    \hdotsfor{4}  \\
    \frac{1}{d+1} & f_{1,d}(x_0)     & \dots & f_{d,d}(x_0)
\end{bmatrix}.$$ 
Each row and each column of~$M$ sums to one; $M$ is doubly stochastic.
By Birkhoff's theorem $M$ is a convex combination of permutation matrices. Since all the entries are nonzero 
in the first column, for each~$i$ there is always a permutation $\pi_i\colon\{0,\dots,d\}\longrightarrow\{0,\dots,d\}$ 
such that $\pi_i(0)=i$ and $f_{j,\pi_i(j)}\ne 0$ for $j=1,\dots,d$. By restricting the domain of $\pi_i$ to $\{1,\dots,d\}$, 
we have a bijection from $\{1,\dots,d\}$ to $\{0,\dots,d\}\setminus\{i\}$ and $f_{j,\pi_i(j)}(x_0)\ne 0$ for $j=1,\dots,d$, 
which implies $x_0 \in C_{\pi_i(1)}^1\cap\dots\cap C_{\pi_i(d)}^d$.	
\end{proof}

By applying Theorem~\ref{thm:strong-colorful-KKM} to the preference sets of $d+1$ people, there is a way to cut the cake such that regardless of which slice the first person chooses, the remaining slices can still be distributed in a way where everyone is satisfied with the division.
Thus, this is a fair cake division where the preference sets of one person are hidden.

\begin{cor}
\label{cor:secret-pref}
	Envy-free cake divisions exist for any number of people even if the preferences of one person are secret.
\end{cor}

\begin{proof}
	For a cake division into $d+1$ parts, the $d$ known preference sets yield $d$ KKM covers of~$\Delta_d$. 
	Find a partition of the cake according to Theorem~\ref{thm:strong-colorful-KKM}. If the person with secret
	preferences chooses slice~$i$, permutation~$\pi_i$ distributes the remaining slices in an envy-free way.
\end{proof}

Corollary~\ref{cor:secret-pref} is the correct analog of the two-person case. Envy-free cake divisions for two 
people exist simply because the first person cuts the cake such that she would be satisfied with either piece,
and the second person chooses her favorite piece. Here the person cutting the cake fairly divides it without 
knowing the preferences of the other person. Corollary~\ref{cor:secret-pref} is an existence result that does 
not inform algorithmic aspects. For recent progress on these aspects see Aziz and Mackenzie~\cite{aziz2016}.

A problem dual to cake division is rental harmony, also treated by Su~\cite{su1999}, where instead of desirables (cake)
one strives to fairly distribute undesirables (rent). We normalize the rent of a $(d+1)$-bedroom apartment to~$1$, and thus
rent divisions among the $d+1$ rooms are parametrized by~$\Delta_d$. The $d+1$ tenants of the apartment try to find a
partition of the rent among the rooms such that tenants can be bijectively assigned to rooms in an envy-free way. Each 
tenant has a collection of $d+1$ preference sets $C_0, \dots, C_{d+1} \subseteq \Delta_d$, where $x \in C_j$ means 
that for rent division~$x$ the tenant (weakly) prefers room~$j$ over the other rooms. We assume that the sets $C_j$ are
closed, and that each tenant prefers a free room over a nonfree one. This leads to the notion of dual KKM cover.

A \emph{dual KKM cover} of the $d$-simplex $\Delta_d$ is a collection of closed subsets 
$C_0, \dots, C_d \subseteq \Delta_d$ such that for any proper subset $I \subseteq \{0, \ldots, d\}$, 
the face of $\Delta_d$ determined by $I$ is covered by $\bigcup_{i \in \{0, \ldots, d\} \setminus I} C_i$
and $\bigcup_{i=0}^d C_i = \Delta_d$. In the language of rent division this means that if we associate the
vertices of~$\Delta_d$ with the $d+1$ rooms then the set $C_i$ where a fixed roommate prefers room~$i$
entirely covers the facet opposite of vertex~$i$. This must be satisfied since the facet opposite of vertex~$i$
corresponds to rent divisions where room~$i$ is free.

Su~\cite{su1999} shows that envy-free rent divisions exist under the given conditions. Here we want to extend
this existence result in the same way as for cake divisions. Even if one tenant does not show up for the meeting
where the division of rent among rooms is decided, and thus her preferences are secret, we still can guarantee
the existence of an envy-free rent division.

\begin{thm}[Strong dual colorful KKM theorem]
\label{thm:strong-dual-colorful-KKM}
	Given $d$ dual KKM coverings $\{C^1_i\}_{i = 0, \dots, d}, \dots,$ $\{C^d_i\}_{i = 0, \dots, d}$ of the $d$-simplex, 
	there is an $x \in \Delta_d$ and there are $d+1$ bijections $\pi_i\colon \{1,\dots,d\}\longrightarrow \{0,\dots,d\}\setminus\{i\}$, 
	$i \in \{0, \dots, d\}$, such that for all~$i$ we have $x \in C_{\pi_i(1)}^1\cap \dots\cap C_{\pi_i(d)}^d$.
\end{thm}

The same proof as for Theorem~\ref{thm:strong-colorful-KKM} works only that now $f$ maps each boundary face $\sigma$
into the opposite face, and thus $f$ is homotopic to the antipodal map on the boundary. In particular, $f$ has nonvanishing degree
on the boundary. As for cake division we now obtain the following corollary:

\begin{cor}
\label{cor:dual-secret-pref}
	Envy-free rent divisions exist for any number of people even if the preferences of one person are secret.
\end{cor}

 
\section{Quantitative Sperner and KKM-type results}
\label{sec:quantitative}

In this section we prove quantitative generalizations of Sperner's lemma and the KKM theorem. That means
that we replace the simplex by an arbitrary polytope or pseudomanifold and establish Sperner-type results
with a lower bound for the number of multicolored facets. As an extension of the KKM theorem we get many $(d+1)$-fold
intersections for certain generalized KKM coverings of $d$-pseudomanifolds. Atanassov~\cite{atanassov1996} conjectured a polytopal
extension of Sperner's lemma, which was proven by De Loera, Peterson, and Su~\cite{deLoera2002}. To state
their result we first need some definitions.

Let $P$ be a $d$-polytope, and let $K$ be a triangulation of $P$, that is, $K$ is topologically a $d$-ball and $\partial K$
is a subdivision of~$\partial P$. A map $c\colon V(K) \longrightarrow V(P)$ from the vertex set of~$K$ to the vertex 
set of~$P$ is called \emph{Sperner coloring} if whenever vertex $v$ of~$K$ is contained in the face $\sigma$ of~$P$
then $c(v)$ is a vertex of~$\sigma$. In this case we refer to $K$ as \emph{Sperner colored}. We think of the vertices of~$P$
as colored by pairwise distinct colors, and the vertices of~$K$ colored by the same color as some vertex of the minimal face of~$P$ it subdivides. 
A facet $\sigma$ of~$K$ (that is, a maximal face) such that $c$ is injective on the vertex set of~$\sigma$ is called \emph{rainbow facet}.

\begin{thm}[De Loera, Peterson, Su~\cite{deLoera2002}]
\label{thm:polytopal-sperner}
	Let $P$ be a $d$-polytope with $n$ vertices, and let $K$ be a Sperner colored 
	triangulation of $P$. Then $K$ has at least $n-d$ rainbow facets.
\end{thm}

For $P = \Delta_d$ this specializes to Sperner's lemma. This was improved by Meunier~\cite{meunier2006} 
to at least $n + \lceil\frac{\delta}{d}\rceil-d-1$ rainbow facets, where $\delta$
is the minimal degree of a vertex in~$P$. Actually, Meunier's bound holds more generally for certain pure polytopal complexes 
(not necessarily homeomorphic to $d$-balls) that are linearly embedded into~$\R^d$. Extensions of Sperner's lemma to PL manifolds
are due to Musin~\cite{musin2014, musin2015}. In particular, using the results of De Loera, Peterson, and Su on 
pebble sets~\cite{deLoera2002} Musin extended Theorem~\ref{thm:polytopal-sperner} to a quantitative Sperner's lemma
for PL manifolds with boundary PL homeomorphic to the boundary of a polytope~\cite[Cor.~3.4]{musin2015}. 

Our first goal is to give a new, entirely elementary, and simple proof of a quantitative Sperner's lemma for simplicial
polytopes. We will identify the boundary of a simplicial polytope as a subdivision of the boundary of a simplex in
several different ways. The classical Sperner's lemma for simplices and multiple-counting then already yields a
quantitative version of Sperner's lemma for simplicial polytopes. Instead of the number of vertices our lower bound 
will depend linearly on the number of facets. This lower bound turns out to be at least as strong as that of Theorem~\ref{thm:polytopal-sperner}
for simplicial $d$-polytopes without vertices of degree~$d$. We will show later that it is superfluous to exclude vertices of degree~$d$.

Gr\"unbaum~\cite{grunbaum1965} showed that the boundary of any polytope is a subdivision of~$\partial\Delta_d$.
Actually his proof also shows that any (simplicial) polytope is a subdivision of~$\partial\Delta_d$ in many different ways. 
We state this stronger version of Gr\"unbaum's theorem below and briefly outline the proof; for details refer to~\cite{grunbaum1965}.

\begin{thm}
\label{thm:subdivision}
	Let $P$ be a simplicial $d$-polytope and $\sigma \subseteq P$ a fixed facet. 
	Then $\partial P$ is a subdivision of~$\partial\Delta_d$ where $\sigma$ is a facet of~$\Delta_d$.
\end{thm}

\begin{proof}
We prove this by induction on~$d$. Let $v$ be an arbitrary vertex of~$\sigma$, and let $\tau$ be the face of~$\sigma$
that contains all vertices except for~$v$. The link of $v$, that is, the intersection of $\partial P$ with a hyperplane close to~$v$,
is the boundary of a simplicial $(d-1)$-polytope and thus a subdivision of~$\partial\Delta_{d-1}$ in such a way that $\tau$ is a 
face of~$\Delta_{d-1}$. The antistar of~$v$ in~$\partial P$, that is, all faces of $\partial P$ that do not contain~$v$, 
is homeomorphic to a $(d-1)$-cell with boundary the link of~$v$. We now obtain $\partial P$ as a subdivision of~$\partial\Delta_d$ 
by coning $\partial \Delta_{d-1}$ in the link of~$v$ with~$v$ and adding the antistar of~$v$.
\end{proof}

Multiple-counting and Sperner's lemma now imply a quantitative version of Sperner's lemma for simplicial polytopes:
Every Sperner colored triangulation of $P$ contains at least $\frac{f_{d-1}(P)}{d+1}$ rainbow facets, where $f_{d-1}(P)$
denotes the number of $(d-1)$-faces of~$P$. In order
to refine this bound further we need an additional lemma. Since our eventual goal is to prove quantitative Sperner-type results
for pseudomanifolds we will phrase and prove the lemma in this generality. 

A $d$-dimensional simplicial complex is called \emph{pure} if every face is contained in a $d$-dimensional face. These top-dimensional
faces are called \emph{facets}. The \emph{dual graph} of a pure $d$-dimensional simplicial complex $K$ has as vertex set the facets of 
$K$ and an edge for any two facets that share a common $(d-1)$-face. If the dual graph is connected $K$ is \emph{strongly connected}.
The \emph{link} of a face $\sigma \in K$ is the subcomplex $\lk(\sigma) = \{\tau \in K \: | \: \sigma \cup \tau \in K, \sigma \cap \tau = \emptyset\}$. 
A pure, $d$-dimensional, strongly connected simplicial complex $K$ where each $(d-1)$-face is contained in precisely two facets and where
all links of faces of codimension at least two are connected is called \emph{pseudomanifold}. We will refer to $d$-dimensional pseudomanifolds 
simply as $d$-pseudomanifolds. 
A \emph{$d$-pseudomanifold with boundary} is defined in the same way, where now each $(d-1)$-face is contained
in either one or two facets. 
Those $(d-1)$-faces contained in exactly one facet induce a subcomplex called the \emph{boundary}. The boundary
$\partial K$ of a pseudomanifold with boundary $K$ is precisely the topological boundary of the geometric realization of~$K$.
We only consider compact pseudomanifolds, that is, every pseudomanifold has only finitely many faces. 

For two faces $\sigma$ and $\tau$ of a simplicial complex we denote by $\tau * \sigma$ their \emph{join}, that is, the face
with vertices~$\tau \cup \sigma$.

\begin{lem}
\label{lem:deg-d}
	Let $K$ be a $d$-pseudomanifold, and let $\sigma_1, \dots, \sigma_{d+1}$ be pairwise distinct facets of $K$ that involve only
	$d+2$ vertices of~$K$. Then $K$ contains a vertex of degree $d+1$.
\end{lem}

\begin{proof}
	Any $1$-pseudomanifold contains only vertices of degree two. We can thus assume that $d \ge 2$. Denote by $W$ the $d+2$
	vertices of $\bigcup_i \sigma_i$. Any $\sigma_i$ contains all but one vertex of $W$, and thus $\bigcap_i \sigma_i$ contains
	precisely one vertex~$v$. The $\sigma_i \setminus v$ are pairwise distinct $d$-element subsets of the $(d+1)$-element set
	$W \setminus v$. Since there are only $d+1$ subsets of size $d$ of $W \setminus v$, we have that the set $\{\sigma_i \setminus v \: | \:
	i = 1, \dots, d+1\}$ consists of all subsets of $W \setminus v$. Thus the link $\lk(v)$ contains~${\partial \Delta_d}$.
	We show that this implies that $\lk(v) \cong \partial \Delta_d$ and thus $v$ has degree~${d+1}$: otherwise, since 
	$\lk(v) \supsetneq \partial \Delta_d$ is connected there is a maximal face $\tau$ of $\partial \Delta_d$ and a vertex 
	$w \in \lk(v) \setminus \partial \Delta_d$ such that $\sigma = \tau * w$ is a face of~$\lk(v)$. The face $\tau$ is at most
	$(d-2)$-dimensional, or else $v * \sigma$ would be a face of dimension at least $d+1$ in~$K$. If $\tau$ was 
	$(d-2)$-dimensional, then $v * \tau$ would be a $(d-1)$-face that is contained in at least three facets, which violates
	the pseudomanifold condition. For lower-dimensional $\tau$ the face $v * \tau$ has codimension at least two and 
	a disconnected link: $\partial \Delta_d \cap \lk(v * \tau)$ is a connected component of $\lk(v * \tau)$ that does not 
	contain~$w$.
\end{proof}

The operation of \emph{stacking} a facet $\sigma$ in a $d$-pseudomanifold consists of replacing $\sigma$ by the join of a new vertex~$v$
and~$\partial\sigma$. The vertex $v$ then has degree $d+1$ and $\lk(v) \cong \partial\Delta_d$. If for some vertex $v$ of degree $d+1$ the
inverse operation of stacking -- \emph{unstacking} -- results in a well-defined $d$-pseudomanifold, we call $v$ \emph{stacking vertex}.
(That is, the boundary of the simplex does not contain stacking vertices, since unstacking is not well-defined.) A simplicial complex is
called \emph{stacked} if it can be obtain from the boundary of the simplex by repeated stacking.
The notions of stacking and stacking vertices are defined for simplicial polytopes as the corresponding notions for its boundary complex.

\begin{thm}
\label{thm:polytopal-sperner2}
	Let $P$ be a simplicial $d$-polytope on $n$ vertices without stacking vertices and $K$ a Sperner colored triangulation of~$P$. Then 
	$K$ has at least $\frac{f_{d-1}(P)-2}{d-1}$ rainbow facets.
\end{thm}

\begin{proof}
	For the $d$-simplex $\Delta_d$ we have that $\frac{f_{d-1}(\Delta_d)-2}{d-1} = 1$ and thus the statement is true by the 
	classical Sperner lemma. Otherwise the fact that $P$ has no stacking vertices implies that $P$ does not have vertices of degree~$d$.
	For each facet $\sigma$ of $P$ we can consider $\partial P$ as a subdivision of~$\partial\Delta_d$, where $\sigma$ is a facet of the 
	simplex. We want to invoke the classical Sperner lemma for each such simplex. This will yield the result by multiple counting.
	
	Fix a facet $\sigma$ of $P$ and let $\Delta_d$ be the simplex given by Theorem~\ref{thm:subdivision} such that $\Delta_d = \sigma * v$
	for some vertex $v \in P \setminus \sigma$ and $\partial P$ subdivides~$\partial \Delta_d$. Recolor the vertices of $K$ to yield
	a Sperner coloring with respect to~$\Delta_d$. We accomplish this by coloring every vertex of $K$ that is not colored by one of the
	colors in~$\sigma$ by the same color as~$v$. Now Sperner's lemma guarantees the existence of a rainbow facet. In particular, this
	facet is also a rainbow facet in the original coloring of $K$ that exhibits all the colors of~$\sigma$.
	
	We repeat this process for each facet $\sigma$ of~$P$. A rainbow facet of $K$ cannot be counted more than $d-1$ times in this way
	by Lemma~\ref{lem:deg-d} since $P$ does not contain vertices with degree~$d$. So the number of rainbow facets in $K$ is at least
	$\frac{f_{d-1}(P)}{d-1} > \frac{f_{d-1}(P)-2}{d-1}$.
\end{proof}

For simplicial polytopes without stacking vertices the bound of Theorem~\ref{thm:polytopal-sperner2} is at least as good as that of
Theorem~\ref{thm:polytopal-sperner}. This follows from Barnette's lower bound theorem~\cite{barnette1973}. 
We have that $\frac{f_{d-1}(P)-2}{d-1} \ge \frac{(d-1)n - (d+1)(d-2)-2}{d-1} = \frac{(d-1)n - d(d-1)}{d-1} = n-d$.

The condition in Theorem~\ref{thm:polytopal-sperner2} that $P$ contain no stacking vertices is superfluous. If $P'$ is obtained
from the simplicial polytope $P$ by stacking an arbitrary facet, then any Sperner colored triangulation of $P'$ has at least one
additional rainbow facet compared to the minimal number of rainbow facets of a Sperner colored triangulation of~$P$. Again, 
we will prove this in greater generality for pseudomanifolds below. In order to obtain quantitative results for pseudomanifolds we need
to employ different methods. Our proof essentially is a path-following argument (one of the standard ways to prove Sperner-type 
results). However, we will first introduce level sets of PL maps as a convenient bookkeeping device for path-following arguments.

Given a smooth map $f\colon M \longrightarrow N$ from a $(d+1)$-manifold $M$ to 
a $d$-manifold~$N$ and a regular value $y \in N$, the fiber $f^{-1}(y)$ consists of 
disjoint paths with endpoints in the boundary~$\partial M$. This preimage theorem remains true in the PL category. This presents
the opportunity of phrasing path following arguments in certain cell complexes by 
considering a facewise affine map and its fibers. 
 
Here we will be interested in facewise affine maps $f\colon K \longrightarrow \R^d$ from a
$d$-dimensional pseudomanifold~$K$ to Euclidean $d$-space and relate the parity of zeros
of~$f$ to the parity of the zeros of the first $d-1$ components of~$f$ on the boundary~$\partial K$
with positive last component. For this the map~$f$ has to be sufficiently non-degenerate, 
that is, the zeros of $f$ are isolated and not in the boundary $\partial K$, and the zeros of the 
first $d-1$ components of $f$ form pairwise disjoint embedded paths and embedded circles in~$K$. 
Denote by $r(f)$ the number of zeros of the map~$f$,
and denote by $r_+(f)$ the number of zeros of $(f_1, \dots, f_{d-1})$ in~$\partial K$ with positive last
component~$f_d$. We will extend a result of Ramos~\cite{ramos1996} relating the parity of $r(f)$ and $r_+(f)$ to one another.
Ramos used these methods to prove results about hyperplane mass partitions.

\begin{lem}
\label{lem:path-following-PL}
	Let $f\colon K \longrightarrow \R^d$ be a non-degenerate face-wise affine map from a $d$-pseudomanifold to~$\R^d$.
	Then $r(f) \equiv r_+(f) \mod 2$.
\end{lem}

\begin{proof}
	The set $Z = \{x \in K \: | \: f_1(x) = \dots = f_{d-1}(x) = 0\}$ consists of embedded circles and embedded paths with their
	endpoints in the boundary~$\partial K$. The embedded circles contain an even number of zeros of~$f$ and thus have no
	influence on $r(f) \mod 2$. We pair the points in $\partial K \cap Z$: the points $x, y \in \partial K \cap Z$ are paired with 
	one another if they are connected by a path in~$Z$. If $f_d(x)$ and $f_d(y)$ have equal sign then the path connecting 
	$x$ and $y$ contains an even number of zeros of~$f$. If the signs of $f_d(x)$ and $f_d(y)$ are opposite then this path 
	contains an odd number of zeros of~$f$. This proves the lemma since either $f_d(x) > 0$ or $f_d(y) > 0$, so precisely 
	one of $x$ or $y$ contributes one to~$r_+(f)$. By non-degenericity $f_d(x) = 0$ or $f_d(y) = 0$ is not allowed.
\end{proof}

We now define what it means for a pseudomanifold to be Sperner colored. 
Let $B$ be a $(d-1)$-pseudomanifold and $K$ a $d$-pseudomanifold with boundary. We say that $K$ is \emph{Sperner colored with respect to~$B$}
if $\partial K$ is a subdivision of~$B$ (where we fix one particular way in which $\partial K$ subdivides~$B$) and there is a \emph{Sperner coloring}
$c\colon V(K) \longrightarrow V(B)$, that is, if vertex $v$ of~$\partial K$ subdivides face $\sigma$ of~$B$ then $c(v)$ is a vertex of~$\sigma$.
We now use path-following via PL maps to prove a quantitative Sperner's lemma for pseudomanifolds. The following lemma is
central.

\begin{lem}
\label{lem:rainbow}
	Let $B$ be a $(d-1)$-pseudomanifold and $K$ a $d$-pseudomanifold with boundary such that $\partial K$ is a subdivision of~$B$. 
	Let $K$ be Sperner colored with respect to $B$. Then for every $(d-1)$-face~$\sigma$ of~$B$ there is a rainbow facet of~$K$ 
	that exhibits the $d$ colors of~$\sigma$ (and one additional distinct color).
\end{lem}

\begin{proof}
	Denote the $d$ colors of the vertices of~$\sigma$ by~${c_1, \dots, c_d}$.
	Suppose the simplex $\Delta_d = \mathrm{conv}\{v_0, v_1, \dots, v_d\}$ is realized in $\R^d$ in such a way that 
	the origin is in the interior of~$\Delta_d$ and that only the relative interior of the face $\mathrm{conv}\{v_1, \dots, v_d\}$ 
	intersects the ray $\{(y_1, \dots, y_d) \in \R^d \: | \: y_1 = \dots = y_{d-1} = 0, y_d > 0\}$ and no other 
	boundary face of $\Delta_d$ intersects this ray. Define an affine map $f\colon K \longrightarrow \Delta_d \subseteq \R^d$ 
	by sending every vertex of color $c_i$ to vertex~$v_i$ of~$\Delta_d$ and every vertex that is not colored by any $c_i$ to 
	vertex~$v_0$ of~$\Delta_d$. The map $f$ is non-degenerate.
	
	Now $r_+(f)$ counts those faces in $\partial K$ that are colored precisely by $c_1, \dots, c_d$ 
	and there are an odd number of those $(d-1)$-faces by the ordinary Sperner lemma, while $r(f)$ 
	counts the number of rainbow facets that use the colors $c_1, \dots, c_d$ and one additional 
	distinct color. The claim now follows from Lemma~\ref{lem:path-following-PL}.
\end{proof}

\begin{thm}
\label{thm:sperner-pseudomfld}
	Let $B$ be a $(d-1)$-pseudomanifold and $K$ a $d$-pseudomanifold with boundary such that $\partial K$ is a subdivision of~$B$. 
	Let $K$ be Sperner colored with respect to $B$. Then $K$ has at least $\frac{f_{d-1}(B)-2}{d-1}$ rainbow facets with pairwise 
	distinct sets of colors. For $d \ge 4$ this bound may hold with equality only if $B$ is stacked. In particular, for $d \ge 4$ the lower bound is not 
	sharp if $B$ is not a sphere.
\end{thm}

\begin{proof}
	We first assume that the degree of any vertex in $B$ is at least~${d+1}$. Using Lemma~\ref{lem:rainbow}
	we pair each $(d-1)$-face $\sigma$ of~$B$ with a rainbow facet of~$K$ that exhibits the colors of~$\sigma$.
	Any rainbow facet of~$K$ gets paired with at most~${d-1}$ faces of~$B$ in this way: suppose a rainbow
	facet $\sigma$ of~$K$ exhibited the $d$ colors of~$d$ faces of~$B$; since $\sigma$ only involves $d+1$ 
	vertices, Lemma~\ref{lem:deg-d} yields a vertex of degree~$d$ in~$B$. We assumed that no such vertex exists.
	Thus the number of rainbow facets in~$K$ is at least $\frac{f_{d-1}(B)}{d-1} > \frac{f_{d-1}(B)-2}{d-1}$.
	
	It is left to show that the lower bound is preserved when we perform stackings on~$B$. Since stacking replaces a
	$(d-1)$-face of~$B$ by $d$ new $(d-1)$-faces, we have to show that stacking increases the lower bound for the 
	number of rainbow facets by one. If $B$ has no stacking vertices then either every vertex has degree at least $d+1$ and 
	the lower bound holds by the first part of the proof, or $B$ is the boundary of the simplex when $\frac{f_{d-1}(B)-2}{d-1} = 1$,
	which is a lower bound for the number of rainbow facets by Sperner's lemma.
	
	Let $K'$ be a $d$-pseudomanifold with boundary that is Sperner colored with respect to~$B'$, where $B'$ is obtained 
	from $B$ by stacking facet $\sigma$ with vertex~$v$. The triangulation $\partial K'$ also subdivides~$B$. Denote the Sperner coloring 
	of $K'$ with respect to $B'$ by $c' \colon V(K') \longrightarrow V(B')$. For each vertex $w$ of $\sigma$ we can modify $c'$ to a Sperner
	coloring of $K'$ with respect to~$B$ by defining $c_w\colon V(K') \longrightarrow V(B)$ by recoloring every vertex of color~$c'(v)$
	with~$c'(w)$. For each Sperner coloring $c_w$ we can find $r = \frac{f_{d-1}(B)-2}{d-1}$ rainbow facets of~$K'$. A rainbow facet 
	of the coloring $c_w$ is also a rainbow facet of~$c'$. We only need to show that the sets of $r$ rainbow facets that we obtain for 
	each coloring $c_w$, are not all the same, and thus there must be at least $r+1$ rainbow facets for~$c'$. If this is not the case, 
	then there are only two possibilities: each rainbow facet (which is a rainbow facet for any~$c_w$) does not contain any vertex~$u$ 
	with $c'(u)  = c'(v)$, or for every vertex~$u$ of every rainbow facet we have that $c'(u) \ne c'(w)$ for every vertex~$w$ of~$\sigma$. 
	Let $\tau$ be a $(d-1)$-face of $B'$ containing~$v$. There is a rainbow facet of $K'$ that uses all colors in~$\tau$ by Lemma~\ref{lem:rainbow}. 
	So both of these cases are impossible. 
\end{proof}

Musin~\cite[Cor.~3.4]{musin2015} proved that any Sperner coloring of a triangulation of a PL manifold with boundary an $n$-vertex 
polytope has at least $n-d$ rainbow facets.
This lower bound is improved by Theorem~\ref{thm:sperner-pseudomfld} (and extended to a larger class of objects). 
Barnette's lower bound theorem holds more generally for manifolds, 
see Kalai~\cite{kalai1987}, and pseudomanifolds, see Fogelsanger~\cite{fogelsanger1988} and Tay~\cite{tay1995}. 
Thus we have that $\frac{f_{d-1}(B)-2}{d-1} \ge n-d$ as before.

Meunier~\cite{meunier2006} proves a slight improvement of the lower bound of $n-d$ for pseudomanifolds with boundary that can be linearly 
embedded into~$\R^d$. In fact, Meunier proves his result for the more general class of \emph{polytopal bodies} where faces are not necessarily
simplices but arbitrary polytopes, where every $(d-1)$-face is contained in one or two facets and the boundary complex is strongly connected.
The boundary complex itself is not required to be a polytope (or even simply connected).

Now that Theorem~\ref{thm:sperner-pseudomfld} provides a quantitative version of Sperner's lemma for pseudomanifolds, and Bapat's result,
Theorem~\ref{thm:colored-sperner}, establishes a colorful version of Sperner's lemma, we must ask for a Sperner-type result that is 
simultaneously quantitative and colorful. Theorem~\ref{thm:colored-sperner-pseudomfld} is this result in a weak sense: it is a colorful 
extension of Theorem~\ref{thm:sperner-pseudomfld}, but for the case of a simplex it specializes to a weaker version of Bapat's result; we do not establish the
existence of a rainbow facet as in Theorem~\ref{thm:colored-sperner}, but of a rainbow face of perhaps lower dimension.
Let $K$ be a $d$-pseudomanifold with boundary such that $\partial K$ subdivides the closed $(d-1)$-pseudomanifold~$B$. 
Let $c_0, \dots, c_d\colon V(K) \longrightarrow V(B)$ be $d+1$ Sperner colorings of $K$ with respect to~$B$. A face $\sigma$
of $K$ is called \emph{rainbow face with respect to $c_0, \dots, c_d$} if the set $\{c_0(v_0), \dots, c_d(v_d)\}$ has cardinality
$d+1$ where $v_0, \dots, v_d$ are not necessarily distinct vertices of~$\sigma$. This condition is weaker than the one shown 
in Theorem~\ref{thm:colored-sperner}.

\begin{thm}
\label{thm:colored-sperner-pseudomfld}
	Let $B$ be a $(d-1)$-pseudomanifold and $K$ a $d$-pseudomanifold with boundary such that $\partial K$ is a subdivision of~$B$. 
	Let $c_0, \dots, c_d\colon V(K) \longrightarrow V(B)$ be $d+1$ Sperner colorings of $K$ with respect to~$B$. Then $K$ has 
	$m \ge \frac{f_{d-1}(B)-2}{d-1}$ subsets $I_1, \dots, I_m \subseteq V(B)$ of cardinality $d+1$ such that for each $1 \le k \le m$
	there is a rainbow face with respect to $c_0, \dots, c_d$ exhibiting the colors in~$I_k$.
\end{thm}

\begin{proof}
	Denote by $K'$ the barycentric subdivision of~$K$. If vertex $v$ of $K'$ subdivides a face $\sigma$ of dimension~$k$ define $c(v) = c_k(w)$,
	where $w$ is an arbitrary vertex of~$\sigma$. Now $c\colon V(K') \longrightarrow V(B)$ is a Sperner coloring of~$K'$ with respect to~$B$.
	By Theorem~\ref{thm:sperner-pseudomfld} there are $\frac{f_{d-1}(B)-2}{d-1}$ rainbow facets that use pairwise distinct sets of colors.
	These correspond to rainbow faces with respect to $c_0, \dots, c_d$.
\end{proof}

The special case of $c_0 = c_1 = \dots = c_d$ is Theorem~\ref{thm:sperner-pseudomfld}. Here a face that exhibits $d+1$ distinct colors
must necessarily be a facet.

Let $K$ be a pseudomanifold with boundary. Identify the $n$ vertices of $\partial K$ with $1, \dots, n$. A covering of $K$ by closed sets
$C_1, \dots, C_n$ such that each face $\sigma$ of $\partial K$ is covered by $\bigcup_{i \in \sigma} C_i$ is called \emph{KKM cover}.
KKM covers were defined by De Loera, Peterson, and Su~\cite{deLoera2002} for polytopes. They established a quantitative KKM theorem
for $d$-polytopes as a corollary of their polytopal Sperner lemma, and thus obtain $n-d$ different $(d+1)$-fold intersections of sets that are
part of the KKM cover. Actually De Loera, Peterson, and Su relate the lower bound for the number of $(d+1)$-fold intersections to 
the covering number of the polytope, that is, the minimal number of $d$-simplices which share all their vertices with the $d$-polytope $P$ and
together entirely cover~$P$.
Here we extend (the simplicial case) to a quantitative KKM theorem for pseudomanifolds with our usual improved
lower bound that depends on the number of $(d-1)$-faces instead of the number of vertices. In fact, we immediately prove a colorful
version of this theorem, which specializes to a ``colorless'' KKM theorem by considering the same KKM cover $d+1$ times.

\begin{thm}
\label{thm:colored-KKM-pseudomfld}
	Let $K$ be a $d$-dimensional pseudomanifold with boundary and suppose the boundary $\partial K$ has $n$ vertices.
	Let $\{C^0_i\}_{i = 1, \dots, n}, \dots, \{C^d_i\}_{i = 1, \dots, n}$ be $d+1$ KKM covers of~$K$. Then there are 
	$m = \frac{f_{d-1}(\partial K) -2}{d-1}$ distinct $(d+1)$-element subsets $I_1, \dots, I_m \subseteq \{1,2,\dots, n\}$ 
	and bijections $\pi_k\colon \{0, \dots, d\} \longrightarrow I_k$ such that
	$\bigcap_{j=0}^d C^j_{\pi_k(j)} \ne \emptyset$ for each $1 \le k \le m$.
\end{thm}

\begin{proof}
	Assume w.l.o.g. that $C^j_i$ is disjoint from the facet opposite the vertex~$i$.
	Let $K'$ be a subdivision of $K$ where each face has diameter at most~$\varepsilon$. Define $d+1$ Sperner colorings
	of $K'$ with respect to~$B$ in the following way: $c_j(v) = i$ if $v \in C^j_i$ but $v \notin C^j_k$ for $k < i$.
	By Theorem~\ref{thm:colored-sperner-pseudomfld} there are $m$ distinct 
	$(d+1)$-element subsets $I_1^\varepsilon, \dots, I_m^\varepsilon \subseteq \{1,2,\dots, n\}$ such that for each $1 \le k \le m$
	there is a rainbow face with respect to $c_0, \dots, c_d$ exhibiting the colors in~$I_k^\varepsilon$.
	In particular, there is for each~$k$ a point $x$ at distance at most~$\varepsilon$ from $C^j_{\pi_k(\ell)}$, $\ell \in \{0, \dots, d\}$,
	where $\pi_k\colon \{0, \dots, d\} \longrightarrow I_k^\varepsilon$ is a bijection.
	For $\varepsilon \longrightarrow 0$ there are at least $m$ subsets $I_1, \dots, I_m \subseteq \{1,2,\dots, n\}$ 
	that are exhibited by rainbow faces with respect to $c_0, \dots, c_d$ infinitely many times. By compactness of~$K$ we have
	the desired result for $\varepsilon \longrightarrow 0$.
\end{proof}


\bibliographystyle{amsplain}

\end{document}